\documentclass[12pt,reqno]{amsart}

\usepackage{a4,url,amssymb,enumerate,tikz,scalefnt}
\usetikzlibrary{positioning}
\usepackage[colorlinks=true,urlcolor=blue,linkcolor=blue,citecolor=blue]{hyperref}

\usepackage{pdfcomment}
\usepackage{amsmath}
\usepackage{amsfonts}
\usepackage{mathtools}
\usepackage{enumitem}

\newcommand{\sts}{\mid}

\numberwithin{equation}{section}
\theoremstyle{plain}
\newtheorem{theorem}{Theorem}[section]
\newtheorem{lemma}[theorem]{Lemma}
\newtheorem{corollary}[theorem]{Corollary}
\newtheorem{claim}{Claim}

\theoremstyle{definition}
\newtheorem{example}[theorem]{Example}
\newtheorem{remark}[theorem]{Remark}

\newenvironment{thmenumerate}{
\begin{enumerate}[label=\textup{(\arabic*)}, widest=(3), leftmargin=9mm,itemsep=1mm]}{
\end{enumerate}}

\newenvironment{txtitemize}{
\begin{itemize}[leftmargin=6mm,itemsep=1mm]}{\end{itemize}}

\newcommand{\N}{\ensuremath{\mathbb{N}}}

\newcommand{\sdp}[2]{\begin{pmatrix} #1\\ #2\end{pmatrix}}

\title[Bounded homomorphisms and fiber products]{Bounded homomorphisms and finitely generated fiber products of lattices}
\author[W.~DeMeo]{William DeMeo}
\author[P.~Mayr]{Peter Mayr}
\address[W. DeMeo and P. Mayr]{Department of Mathematics, University of Colorado,  Boulder, USA}
\email{williamdemeo@gmail.com, peter.mayr@colorado.edu}

\author[N.~Ru\v{s}kuc]{Nik Ru\v{s}kuc}
\address[N.~Ru\v{s}kuc]{School of Mathematics and Statistics, University of St Andrews, St Andrews, Scotland, UK}
\email{nik.ruskuc@st-andrews.ac.uk}

\thanks{The first and second authors were supported by the National Science Foundation under Grant No. DMS 1500254.}

\date{\today}

\begin{document}
\maketitle

\begin{abstract}
 We investigate when fiber products of lattices are finitely generated and obtain a new characterization of bounded
 lattice homomorphisms onto lattices satisfying a property we call
 Dean's condition \eqref{eq:D} which arises from Dean's solution to the word problem for finitely presented lattices.
 In particular, all finitely presented lattices and  those satisfying Whitman's condition satisfy \eqref{eq:D}.
 For lattice epimorphisms $g\colon A\to D$, $h\colon B\to D$, where $A$, $B$ are finitely generated and $D$ 
 satisfies \eqref{eq:D}, we show the following:
 If $g$ and $h$ are bounded, then their fiber product (pullback) $C=\{(a,b)\in A\times B\sts g(a)=h(b)\}$ is finitely
 generated. While the converse is not true in general, it does hold when $A$ and $B$ are free.
 As a consequence we obtain an (exponential time) algorithm to decide boundedness for finitely presented lattices
 and their finitely generated sublattices satisfying \eqref{eq:D}.
 This generalizes an unpublished result of Freese and Nation.
\end{abstract}

\section{Introduction}
\label{sec:i}

 A \emph{subdirect product} of algebraic structures $A$ and $B$ is a subalgebra $C$ of the direct product $A\times B$
 that projects onto both factors.
 In~\cite{MR:FPD,MR:GSP} the second and third authors studied conditions under which direct and subdirect products of
 various algebras are finitely generated. \emph{Direct} products of finitely generated lattices are finitely
 generated. On the other hand,
 note that a congruence $\alpha$ of an algebra $A$ is a subdirect product of two copies of $A$. If $\alpha$
 is finitely generated as subalgebra of $A^2$, then it is clearly finitely generated as congruence of $A$ as well.
 For every non-finitely presented quotient $F(X)/\rho$ of a finitely generated free lattice $F(X)$,
 the congruence $\rho$ is a \emph{subdirect} product of $F(X)$ with itself that is not finitely generated as a lattice.
 In \cite[Example 7.5]{MR:GSP} an explicit example is given of a congruence $\rho$ such that $F(X)/\rho$ is finite,
 hence finitely presented, but $\rho$ is not finitely generated as a sublattice of $F(X)\times F(X)$.
 The present paper is a continuation of that work.
 
 We start by recalling a standard method for constructing subdirect products. Let $A$, $B$ be algebras with epimorphisms $g\colon A\rightarrow D$ and $h\colon B\rightarrow D$
 onto the same homomorphic image $D$. Then the subalgebra
\[ C := \{(a, b)\in A \times B \sts g(a) = h(b)\} \]
 of $A \times B$ is called a \emph{fiber product} (or \emph{pullback}) of $g$ and $h$. Clearly $C$ is a subdirect
 product of $A$ and $B$. 
 Note that when $B=A$ and $h=g$ the resulting fiber product is precisely the kernel of $g$ as a subdirect product in
 $A\times A$.  
 
 Whether a fiber product of lattices is finitely generated turns out to be connected to the following
 properties of homomorphisms that originally appeared in the work of 
 McKenzie on 
 lattice varieties~\cite{MR0313141} and of 
 J\'onsson on free lattices~\cite{MR0472614}.
 Let $A,D$ be lattices. A homomorphism $g\colon A\to D$ is \emph{lower bounded} if for each $d\in D$
 the set $\{ x\in A \sts g(x) \geq d \}$ is either empty or has a least element
  (dually, $\{ x\in A \sts g(x) \leq d \}$ is empty or has a greatest element for \emph{upper bounded}).
 If $g$ is surjective, this condition is equivalent to the preimage $g^{-1}(d)$ having a least element
  (dually, greatest for upper bounded) for each $d\in D$.
 Further $g$ is \emph{bounded} if it is both lower and upper bounded.

 The existence of a lower bounded epimorphism from a free lattice has a strong universal consequence.
 By~\cite[Theorem 2.13]{MR1319815} the following are equivalent for any finitely generated lattice $D$:
\begin{thmenumerate}
    \item There exists a finite set $X$ and a lower bounded epimorphism $f \colon F(X)  \rightarrow  D$
  from the free lattice $F(X)$ onto $D$.
  \item For every finitely generated lattice $A$, every homomorphism $h \colon A \to D$ is lower bounded.
\end{thmenumerate}
  
 If $D$ satisfies one, and hence both, conditions (1),(2) above we say that $D$ is \emph{lower bounded}.
 Of course, the duals of these statements also hold and define \textit{upper bounded lattice}.
 A lattice that is both upper and lower bounded is said to be \textit{bounded}.

 We say a lattice $D$ with finite generating set $P$ satisfies \emph{Dean's condition}~\eqref{eq:D}
 if for all finite subsets
 $S,T\subseteq D$, 
\begin{equation} \label{eq:D} \tag{D}
  \bigwedge S \leq \bigvee T \ \Rightarrow\ \begin{cases} \exists s\in S\colon s\leq \bigvee T \text{ or }\\
  \exists t\in T\colon \bigwedge S\leq t \text{ or } \\
  \exists p\in P\colon \bigwedge S\leq p \leq \bigvee T. \end{cases}
\end{equation}
 Every finitely presented lattice satisfies Dean's condition~\eqref{eq:D} for an appropriate generating set.
 More precisely every lattice $F(P)$ that is freely generated by a finite partial lattice $P$ satisfies~\eqref{eq:D}
 for $P$~\cite[Theorem 2-3.4]{FN:FFP}, and~\cite[Section 2-3.1]{FN:FFP} gives a translation from
 finite presentations into finite partial lattices and conversely.
 Further every finitely generated lattice satisfying \emph{Whitman's condition}~\eqref{eq:W},
\begin{equation} \label{eq:W} \tag{W}
 \bigwedge S \leq \bigvee T \ \Rightarrow\ \exists s\in S\colon s\leq \bigvee T \text{ or } \exists t\in T\colon \bigwedge S\leq t,
\end{equation}
 also satisfies Dean's condition~\eqref{eq:D} for any finite generating set.

 Our first result is that boundedness is a sufficient condition for finite generation of fiber products:

\begin{theorem} \label{thm:fiber}
 Let $A, B, D$ be finitely generated lattices, and assume $D$ satisfies Dean's condition~\eqref{eq:D}.
 If $g\colon A\rightarrow D$ and $h\colon B \rightarrow D$ are bounded epimorphisms,
 then their fiber product is a finitely generated sublattice of $A \times B$.
\end{theorem}

 Theorem~\ref{thm:fiber} will be proved in Section~\ref{sec:f}. Its specialization for $D$ a finitely presented
 lattice resembles a result in congruence permutable varieties: there, for all finitely generated algebras $A, B$
 and finitely presented $D$, every fiber product of epimorphisms $g\colon A\rightarrow D$ and $h\colon B \rightarrow D$
 is finitely generated~\cite[Proposition 3.3]{MR:GSP}; for groups see also~\cite{BM:SFP}.
 
 Finitely generated lattices satisfying Whitman's condition~\eqref{eq:W} are not necessarily finitely presented.
 See McKenzie's example~\cite[Example 2-9.1]{FN:FFP} for a finitely generated but not finitely presented sublattice
 $L_1$ of a finitely presented lattice. This lattice $L_1$ satisfies Whitman's condition as was pointed out to us
 by Freese and Nation.
 However the bounded finitely generated lattices satisfying Whitman's condition~\eqref{eq:W} are exactly the projective
 finitely generated lattices by Kostinsky's Theorem~\cite[Corollary 5.9]{MR1319815}. Hence Theorem~\ref{thm:fiber}
 specialized to $D$ with~\eqref{eq:W} yields a direct proof, and in fact a strengthening, of the known result that
 finitely generated projective lattices are finitely presented. Indeed,
 if $F(X)/\rho$ with $X$ finite is bounded and satisfies Whitman's condition~\eqref{eq:W}, then $\rho$ is
 finitely generated as sublattice of $F(X)\times F(X)$, hence also as a congruence of $F(X)$.

 The following example shows that Dean's condition~\eqref{eq:D} for $D$ cannot be omitted in Theorem~\ref{thm:fiber}:

\begin{example}
 Let $h\colon F(x_1,x_2,x_3,y_1,y_2,y_3) \to F(x_1,x_2,x_3) \times F(y_1,y_2,y_3)$ with
 $h(x_i) := (x_i,0)$ and $h(y_i) := (0,y_i)$ for $i\leq 3$ be the natural epimorphism from the
 free lattice over $6$ generators to the direct square of the free lattice over $3$ generators. Note that the latter
 is bounded since it is a direct product of bounded (free) lattices.
 However $F(x_1,x_2,x_3) \times F(y_1,y_2,y_3)$ is not finitely presented by~\cite[Theorem 3.10]{MR:FPD}.
 Hence bounded and finitely generated does not imply finitely presented. Moreover $\ker h$ is not finitely generated
 as a congruence of $F(x_1,x_2,x_3,y_1,y_2,y_3)$ and in particular not finitely generated as a lattice.
\end{example}
 
 The converse of Theorem~\ref{thm:fiber} is not true in general as will be shown in Section~\ref{sec:c}:

\begin{theorem}
\label{thm:exexists}
 There exists a finitely generated lattice $M$ and an unbounded epimorphism $h\colon M\rightarrow L$ onto
 a finite lattice $L$ such that the kernel of $h$ is finitely generated as a sublattice of $M\times M$.
\end{theorem}

 However, the converse does hold for fiber products of free lattices and, more generally, 
 of lattices that are generated by join prime and by meet prime elements and satisfy Whitman's condition~\eqref{eq:W}.
 The latter include in particular all lattices that are freely generated by some ordered
 set~\cite[Theorem 5.19]{MR1319815}. 

\begin{theorem} \label{thm:prime}
 Let $A,B$ be lattices that satisfy Whitman's condition~\eqref{eq:W} and are generated by join prime elements as well
 as by meet prime elements, let $D$ be a lattice. If the fiber product of epimorphisms $g\colon A \rightarrow D$
 and $h\colon B \rightarrow D$ is a finitely generated sublattice of $A\times B$, then $g$ and $h$ are bounded.
\end{theorem}

 Theorem~\ref{thm:prime} will be proved in Section~\ref{sec:m}. Together with Theorem~\ref{thm:fiber}
 it yields the following characterization of finitely generated fiber products of free lattices (solving Problem 7.6
 of~\cite{MR:GSP}) as well as a new characterization of finitely presented bounded lattices.
 
 \begin{corollary} \label{cor:bounded}
 For any finitely generated lattice $D$ that satisfies Dean's condition~\eqref{eq:D} the following are equivalent:
  \begin{thmenumerate}
  \item \label{it:Dbounded}
    $D$ is bounded.
  \item \label{it:kerfg}
    There exists a finite set $X$ and an epimorphism $h \colon F(X)  \rightarrow  D$ from the free lattice $F(X)$ onto $D$ such that $\ker h$ 
    is a finitely generated sublattice of $F(X) \times F(X)$.
  \item \label{it:allkerfg}
    For every finitely generated lattice $A$ and epimorphism $g \colon A \rightarrow  D$, the kernel of
     $g$ is a finitely generated sublattice of $A\times A$.
  \item \label{it:fiberfg}
    For all finitely generated lattices $A, B$, epimorphisms $g \colon A \rightarrow  D$  and $h \colon B \rightarrow D$, the fiber product of $g$ and $h$ is a finitely generated sublattice of $A \times B$.
\end{thmenumerate}
  \end{corollary}

 Note that~\ref{it:Dbounded}$\Rightarrow$\ref{it:fiberfg} follows from Theorem~\ref{thm:fiber};
 \ref{it:fiberfg}$\Rightarrow$\ref{it:allkerfg}$\Rightarrow$\ref{it:kerfg} are immediate;
 \ref{it:kerfg}$\Rightarrow$\ref{it:Dbounded} follows from Theorem~\ref{thm:prime}.

 The analogous question of characterizing finite generation of fiber products of free semigroups and monoids was
 considered by Clayton~\cite{AC19} in the case where the common quotient $D$ is finite or free.

 Comparing lattices and congruence permutable varieties again, recall that every finitely
 generated congruence $\rho$ of a finitely generated Mal'cev algebra $A$ is finitely generated as a subalgebra of
 $A\times A$.
 By Corollary~\ref{cor:bounded}\ref{it:Dbounded}$\Rightarrow$\ref{it:allkerfg} this also holds for congruences $\rho$ of lattices $A$ whenever
 $A/\rho$ is bounded.
 In particular for every bounded finitely presented lattice $F(X)/\rho$ with $X$ finite, $\rho$ is not only a
 finitely generated  congruence of the free lattice $F(X)$ but also finitely generated as a sublattice of
 $F(X)\times F(X)$.

 
 In the course of proving Theorem~\ref{thm:fiber} we obtain that an epimorphism $g\colon A\to F(P)$ is bounded
 if and only if the pre-images of $P$ under $g$ are bounded; see Corollary~\ref{cor:lb}.
   
 Combining this with~\cite[Section 2-3.1]{FN:FFP} yields~\ref{it:FXbounded}$\Leftrightarrow$\ref{it:Pbounded} in
 the following reformulation of Corollary~\ref{cor:bounded} for finitely presented lattices $F(X)/\rho$.

\begin{corollary} \label{cor:fpbounded}
 Let $X$ be finite. Then the following are equivalent for a congruence $\rho$ of $F(X)$ generated as a congruence  
 by a finite set $R\subseteq F(X)\times F(X)$:
\begin{thmenumerate}
\item \label{it:FXbounded}
 $F(X)/\rho$ is bounded.
\item $\rho$ is finitely generated as a sublattice of $F(X)\times F(X)$.
\item \label{it:Pbounded}
If $u$ is an element of $X$ or a subterm of a term occurring in $R$, the class $u/\rho$ is bounded in $F(X)$.
\end{thmenumerate}
\end{corollary}

 By Corollary~\ref{cor:lb} boundedness for $D$ with finite generating set $P$ satisfying Dean's condition~\eqref{eq:D}
 is determined by pre-images of $P$. This yields an algorithm for deciding
 whether certain types of such lattices $D$ are bounded which we will describe in Section~\ref{sec:algorithm}.
 The assertion of the following theorem for the class of finitely presented lattices was already known
 to Freese and Nation; see~\cite[page 251]{MR1319815}.
 
\begin{theorem} \label{thm:decidable}
 Lower boundedness is decidable for finitely presented lattices and their finitely generated sublattices
 satisfying Dean's condition~\eqref{eq:D}.
\end{theorem}

 Finally let us add a small observation about subdirect products of lattices that are not fiber products but closely
 related:

 \begin{remark}
 If the fiber product of epimorphisms $g\colon A\to D$ and $h\colon B\to D$ is finitely generated, then also
\[ C = \{ (a,b)\in A\times B \sts g(a)\leq h(b) \} \]
 is finitely generated.  
 Indeed if the  fiber product of $g$ and $h$ is generated by $G$, then $C$ is generated by $G \cup \{(0_A,1_B)\}$.
\end{remark}

\section{Bounded homomorphisms imply finite generation}
\label{sec:f}

 To facilitate inductive proofs on the complexity of lattice elements over some generating set,
 we adapt the notation from~\cite[Section II.1]{MR1319815}. 

 Let $A$ be a lattice with finite generating set $X$. For a subset $W$ of $A$ we define
 \begin{align*}
 W^\wedge & := \bigl\{ \bigwedge U \:\bigl|\: U \text{ is a finite subset of } W \bigr\}, \\
 W^\vee & := \bigl\{ \bigvee U \:\bigl|\: U \text{ is a finite subset of } W \bigr\},
\end{align*} 
with the convention that $1:= \bigvee X = \bigwedge \emptyset$ and $0:=\bigwedge X = \bigvee\emptyset$ in $A$.
Next define an ascending chain
\[
X = G_{X,0} \subseteq H_{X,0}\subseteq G_{X,1}\subseteq H_{X,1}\subseteq\dots
\]
of subsets of $A$ inductively as follows:
\begin{equation}
\label{eq:HKdef}
 G_{X,0} := X,\quad H_{X,k}:=G_{X,k}^\wedge,\quad G_{X,k+1}:=H_{X,k}^\vee\quad \text{for } k\in\N.
\end{equation}
 If the generating set $X$ is clear from the context, we write simply $G_k$, $H_k$. 
 Note that $G_k$ for $k\geq 1$ is join-closed, $H_k$ for $k\geq 0$ is meet-closed and
\[ 
 A =\bigcup_{k\in\N} G_k = \bigcup_{k\in\N} H_k. 
\] 
 For $g\colon A\rightarrow D$ an epimorphism, $k\in\N$ and $d\in D$, we define 
\begin{equation}
\label{eq:bela}
 \alpha_{g,k}(d):= \bigvee \{ w\in G_k \sts g(w) \leq d \}, \quad
 \beta_{g,k}(d) := \bigwedge \{ w\in H_k \sts g(w) \geq d \}. 
\end{equation}
 Note that $\alpha_{g,0}(d)$ is not necessarily contained in $G_{X,0} = X$, but for $k\geq 1$ the element
 $\alpha_{g,k}(d)$ is the greatest $w\in G_{X,k}$ with $g(w)\leq d$. Dually $\beta_{g,k}(d)$ is the least
 $w\in H_k$ with $g(w)\geq d$ for all $k\in\N$. Furthermore \eqref{eq:HKdef} yields for all $k \geq 1$
\begin{equation}
\begin{split}
\label{eq:albealt}
&\alpha_{g,k}(d)= \bigvee \{ w\in H_{k-1} \sts g(w) \leq d \}, \\
&\beta_{g,k}(d) = \bigwedge \{ w\in G_{k} \sts g(w) \geq d \}. \\
\end{split}
  \end{equation}
 If the epimorphism $g$ is clear from the context, we write $\alpha_k$ instead of $\alpha_{g,k}$, etc.
 Note that $\alpha_k,\beta_k$ depend on the choice of the generating set $X$ of $A$.

 In~\cite[Section II.1]{MR1319815} $H_k$ and $\beta_k$ are defined exactly as above. We have introduced the
 non-standard notions of $G_k$ and $\alpha_k$ for our proof of Theorem~\ref{thm:fiber}. 

\begin{remark}      
\label{re:duality}
 There is a duality between $H_k$ and $\beta_k$ on one hand and $G_k$ and $\alpha_k$ on the other.
 However when referring to this duality one needs to bear in mind the following:
\begin{txtitemize}
\item
 $G_0 = X$ and $H_0 = X^\wedge$ are not dual to each other with respect to $X$.
 Still we will obtain completely dual formulas for $\alpha_k$ and $\beta_k$ for all $k\in\N$ in
 Lemma~\ref{lem:beta}\ref{it:aka0},\ref{it:bkb0} below.
\item
 By~\eqref{eq:albealt}, if a statement about $G_k$ and $\alpha_k$ also refers to $H_{k-1}$ and $\beta_{k-1}$,
 then its dual for $H_k$ and $\beta_k$ will need to refer to $G_k$ and $\alpha_k$ (and \emph{not} to $G_{k-1}$ and $\alpha_{k-1}$).
\end{txtitemize}
\end{remark}

 In the following lemma we record some basic properties of our functions that extend those given for $\beta_k$
 in~\cite[Theorems 2.2, 2.4]{MR1319815}.

\begin{lemma} \label{lem:beta}
 Let $g\colon A\rightarrow D$ be a lattice epimorphism,
 $a\in A, d,e\in D$, and $k,\ell\in\N$. The following hold: 
 \begin{thmenumerate}
  \item \label{it:mon}
  If $d\leq e$, then $\alpha_k(d)\leq \alpha_k(e)$ and $\beta_k(d)\leq \beta_k(e)$.
  \item\label{it:desc}
  If $k\leq\ell$, then $\alpha_k(d)\leq \alpha_\ell(d)$ and $\beta_k(d)\geq \beta_\ell(d)$.
  \item\label{it:geqbeta}
  If $d\leq g(a)\leq e$, then $\beta_m(d)\leq a\leq\alpha_{m+1}(e)$ for some $m\in\N$.
\item \label{it:aka0}
 $ \alpha_0(d) = \bigvee \{  x\in X  \sts g(x) \leq d \}$
  and for $k>\ell$
\[ \alpha_k(d) = \bigvee \bigl\{ \bigwedge U  \sts U \subseteq G_{k-1},  g(\bigwedge U) \leq d \text{ but } g(u) \not\leq d \text{ for every } u\in U \bigr\} \vee \alpha_\ell(d). \]
\item \label{it:bkb0}
  $\beta_0(d) = \bigwedge\{x\in X \sts g(x) \geq d\}$
  and for $k>\ell$
 \[ \beta_k(d) = \bigwedge \bigl\{ \bigvee U  \sts U \subseteq H_{k-1},  g(\bigvee U) \geq d \text{ but } g(u) \not\geq d \text{ for every } u\in U \bigr\} \wedge \beta_\ell(d). \]
  \item \label{it:Vbk}
 For each non-empty finite $E\subseteq D$
 \[
 \bigwedge \alpha_{k-1}(E)  \leq \alpha_{k}(\bigwedge E), \quad \bigvee \beta_{k-1}(E)  \geq \beta_{k}(\bigvee E).
 \]    
 \end{thmenumerate}
 \end{lemma}  
 
\begin{proof}
 Parts~\ref{it:mon}, \ref{it:desc} and \ref{it:geqbeta} are immediate from the definitions. 

 Item~\ref{it:aka0} is proved by induction on $k$. 
 The base case $k=0$ is just the definition of $\alpha_0(d)$.
 Let $k\geq 1$.
 From \eqref{eq:albealt}, $H_{k-1}=G_{k-1}^\wedge$ and $\alpha_k(d)\geq \alpha_{k-1}(d)$ we have
  \begin{equation}\label{eq:1000}
  \alpha_k(d)=\bigvee\bigl\{ \bigwedge U\sts U\subseteq G_{k-1},
  g(\bigwedge U)\leq d\bigr\} \vee \alpha_{k-1}(d).
  \end{equation}
  If a meetand $u\in U \subseteq G_{k-1}$ satisfies $g(u)\leq d$, then $\bigwedge U\leq u\leq \alpha_{k-1}(d)$, in which case
  $(\bigwedge U)\vee\alpha_{k-1}(d)=\alpha_{k-1}(d)$ and $\bigwedge U$ can be removed from the join in \eqref{eq:1000}.
  So we are left with
  \[ 
  \alpha_k(d) = \underbrace{\bigvee \bigl\{ \bigwedge U  \sts U \subseteq G_{k-1},  g(\bigwedge U) \leq d \text{ but } g(u) \not\leq d \text{ for every } u\in U \bigr\}}_{=:b} \vee \alpha_{k-1}(d). \]
 For $k-1=\ell$ this is the assertion already. 
 Else for $k-1>\ell$ the induction assumption yields
\[
  \alpha_{k-1}(d)= \underbrace{\bigvee \bigl\{ \bigwedge U  \sts U \subseteq G_{k-2},  g(\bigwedge U) \leq d \text{ but } g(u) \not\leq d \text{ for every } u\in U \bigr\}}_{=:c} \vee \alpha_{\ell}(d).
\]
 Since $G_{k-2}\subseteq G_{k-1}$, we have $c\leq b$ and $\alpha_k(d) = b\vee \alpha_{\ell}(d)$ as required. 

 Item~\ref{it:bkb0} is dual to \ref{it:aka0}. The only place in which their proofs differ is that the base case
 for $k=0$ follows from~\eqref{eq:albealt} instead of the definition. 

 For~\ref{it:Vbk}, note that $\beta_{k-1}(d) \in H_{k-1}$ and $g\beta_{k-1}(d)\geq d$ for all $d\in D$. Then
 $w := \bigvee \{ \beta_{k-1}(d) \sts d\in E \}$ is in $H_{k-1}^\vee$ and hence in $H_{k}$. Since
 \[ g(w) = \bigvee \{ g\beta_{k-1}(d) \sts d\in E \} \geq \bigvee E, \]
 we have $w\geq \beta_k(\bigvee E)$ and the claim follows. The proof of the other assertion is dual.
\end{proof}

 Assume that $g\colon A\rightarrow D$ is a lower bounded epimorphism. We denote the least element in the
 preimage of $d\in D$ by
\[
    \beta_g(d) := \bigwedge g^{-1}(d).
\]
 Note that for every $d\in D$ there exists $k\in\N$ such that $\beta_g(d) = \beta_{g,k}(d)$.

 Dually for an upper bounded epimorphism $g\colon A\rightarrow D$ the greatest element in the
 preimage of $d\in D$ is denoted by
\[
    \alpha_g(d) := \bigvee g^{-1}(d).
\] 
 It is not hard to see that $\beta_g$ preserves joins and $\alpha_g$ preserves meets~\cite[page 27]{MR1319815}.
 In general $\alpha_{g,k},\beta_{g,k}$ do not preserve any lattice operations. 
 Still we can obtain some useful
 identities when $D$ satisfies Dean's condition~\eqref{eq:D}.
 
\begin{lemma} \label{lem:betak}
 Let $A$ be a lattice with finite generating set $X$, let $D$ be a lattice with finite generating set $P$
 satisfying Dean's condition~\eqref{eq:D}, and let $g\colon A\rightarrow D$ be an epimorphism.
 Assume every $p\in P$ has a least pre-image $\beta(p) := \bigwedge g^{-1}(p)$ in $A$ and $\beta(p)\in H_{X,0}$. 
 Let $k\in\N$. Then
\begin{thmenumerate}
\item \label{it:betakwedge}
 $\beta_{k}(\bigwedge E) = \bigwedge \beta_{k}(E) \wedge \beta_0(\bigwedge E)$ for each finite $E\subseteq D$;
\item  \label{it:betakbeta}
 $\beta_{k}(d) = \beta(d)$ for all $d\in H_{P,k}$.
\end{thmenumerate}
\end{lemma}

\begin{proof} \ref{it:betakwedge}
 For $k=0$ the statement is immediate from the monotonicity of $\beta_0$ by Lemma~\ref{lem:beta}\ref{it:mon}.
 Let $k\geq 1$. For $d := \bigwedge E$ let $w\in H_{X,k-1}^\vee$ be one of the meetands in the formula for
 $\beta_{k}(d)$ in Lemma~\ref{lem:beta}\ref{it:bkb0}; specifically, $w = \bigvee U$ for some
 $U \subseteq H_{X,k-1}$, where $g(\bigvee U) \geq d$ and $g(u) \not\geq d$ for all $u\in U$.
 We claim that
\begin{equation} \label{eq:bkE}
 w \geq \bigwedge \beta_{k}(E) \wedge \beta_0(d).
\end{equation}  
By Dean's condition~\eqref{eq:D}, the assumption $g(w) = g(\bigvee U) \geq \bigwedge E = d$ yields $g(u) \geq d$
for some $u\in U$ or $g(w) \geq e$ for some $e\in E$ or $g(w) \geq p \geq d$ for some $p\in P$. We consider each case.
 
 {\bf Case 1:}
 $g(u) \geq d$ for some $u\in U$ contradicts our assumption on $w$.
 
 {\bf Case 2:}
 $g(w) \geq e$ for some $e\in E$ yields $w\geq \beta_{k}(e) \geq \bigwedge \beta_{k}(E)$.

 {\bf Case 3:}
 Assume $g(w) \geq p \geq d$ for some $p\in P$. Then 
\begin{align*}
 w & \geq \beta(p) & \\  
   & = \beta_0(p) & \text{by the  assumption } \beta(p)\in H_{X,0} \\
   & \geq \beta_0(d) & \text{by Lemma~\ref{lem:beta}\ref{it:mon}}.
\end{align*}     
 In any case we obtain~\eqref{eq:bkE}. Thus by Lemma~\ref{lem:beta}\ref{it:bkb0} we have
\[ \beta_{k}(d) \geq \bigwedge \beta_{k}(E) \wedge \beta_0(d). \]
 The converse inequality follows from  Lemma~\ref{lem:beta}\ref{it:mon},\ref{it:desc}.

 \ref{it:betakbeta} We use induction on the complexity of $d\in D$ over the generating set $P$.
 The base case is the assumption that $\beta_0(d) = \beta(d)$ for all $d\in G_{P,0} = P$.
 
 For $k\in\N$ we will use two induction steps alternatingly: from $G_{P,k}$ to $H_{P,k}$ and from $H_{P,k}$ to
 $G_{P,k+1}$.
 First assume that $\beta_k(e)=\beta(e)$ for all $e\in G_{P,k}$.
 Let $d\in H_{P,k} = G_{P,k}^\wedge$ and write $d=\bigwedge E$ for $E\subseteq G_{P,k}$. 
 For $\ell\in\N$,
\begin{align*}
 \beta_{k+\ell}(d) & = \bigwedge\beta_{k+\ell}(E) \wedge \beta_0(d) & \text{by item~\ref{it:betakwedge}} \\
     & = \bigwedge\beta_k(E) \wedge \beta_0(d) & \text{by induction assumption} \\
     & = \beta_k(d) & \text{by item~\ref{it:betakwedge}.}
\end{align*}  
 Hence $\beta_{k}(d)$ is the least element in $A$ that $g$ maps to $d$ and $\beta_{k}(d) = \beta(d)$ for all
 $d\in H_{P,k}$.
 
 Next assume that $\beta_k(e)=\beta(e)$ for all $e\in H_{P,k}$. Let $d\in G_{P,k+1} = H_{P,k}^\vee$ and write
 $d=\bigvee E$ for $E\subseteq H_{P,k}$. Then
\begin{align*}
 \beta(d) & = \bigvee \beta(E) & \text{since } \beta \text{ preserves joins} \\
          & = \bigvee \beta_{k}(E) & \text{by induction assumption} \\
          & \geq \beta_{k+1}(d) & \text{by Lemma~\ref{lem:beta}\ref{it:Vbk}.}
\end{align*}
 Since the converse inequality holds trivially, $\beta_{k+1}(d) = \beta(d)$ for all $d\in G_{P,k+1} \subseteq H_{P,k+1}$.
 This concludes both induction steps and the proof of~\ref{it:betakbeta}.
\end{proof}

 Lemma~\ref{lem:betak}\ref{it:betakbeta} yields in particular:

\begin{corollary} \label{cor:lb}
 Let $g\colon A\to D$ be an epimorphism from a finitely generated lattice $A$ onto a lattice $D$
 with finite generating set $P$ satisfying Dean's condition~\eqref{eq:D}.
 Then $g$ is lower bounded if and only if the pre-images under $g$ of $P$ are lower bounded.
\end{corollary}

 We are now ready to prove Theorem~\ref{thm:fiber}.

\begin{proof}[Proof of Theorem~\ref{thm:fiber}]
 Let $A,B,D$ be lattices with finite generating sets $X,Y,P$, respectively,
 let $D$ and $P$ satisfy Dean's condition~\eqref{eq:D},
 let $g\colon A\rightarrow D$ and $h\colon B \rightarrow D$ be bounded epimorphisms,
 and let $E := g(X)\cup h(Y) \cup P$.
 
 We will show that the fiber product  
 $$C:=\Bigl\{ \sdp{a}{b}\in A\times B\sts g(a)=h(b)\Bigr\}$$
 is generated by the finite set
 \[
  Z :=
  \Bigl\{ \sdp{x}{\alpha_{h}g(x)}, \sdp{\beta_{g}h(y)}{y}, \sdp{\alpha_g(d)}{\beta_{h}(d)}, \sdp{\beta_g(d)}{\alpha_h(d)} \sts x\in X, y\in Y, d\in E \Bigr\}.
 \]
 Enlarging the original generating sets $X,Y$ by finitely many elements if necessary, we may assume that
 $X$ and $Y$ actually are the projections of $Z$ onto its first and second components, respectively.  
   
 We proceed via a series of technical claims. We begin by observing that the following hold from the definition of $Z$:
\begin{enumerate}[label=\textup{(Z\arabic*)}, widest=(3), leftmargin=10mm,itemsep=1mm]
\item \label{it:G1}
 For all $x\in X$ we have $(x,\alpha_hg(x))$ and $(x,\beta_hg(x))\in Z \cap (X\times Y)$.
\item \label{it:G2}
Let $k\in \N$. For all
$a_1\in G_{X,k}$, $a_2\in H_{X,k}$
there exist
$b_1\in G_{Y,k}$, $b_2\in H_{Y,k}$, such that
$(a_1,b_1),(a_2,b_2)\in \langle Z\rangle$.
\item \label{it:G3}
For every $k\in\N$ we have
$g(G_{X,k})=h(G_{Y,k})$ and $g(H_{X,k})=h(H_{Y,k})$.
\end{enumerate}
Of course the symmetric versions of statements \ref{it:G1}, \ref{it:G2} with components (as well as $X$ and $Y$) swapped
hold as well.

 \begin{claim}
 \label{cl:1}
 The following hold:
 \begin{thmenumerate}
 \item \label{it:cl11}
 $\forall b\in B\ \exists b'\leq b$: $\sdp{\alpha_{g,0}h(b)}{b'}\in\langle Z\rangle$,
 \item \label{it:cl13}
 $\forall a\in A\ \exists a'\geq a$: $\sdp{a'}{\beta_{h,0}g(a)}\in\langle Z\rangle$.
 \end{thmenumerate}
  \end{claim}
  
  \begin{proof} 
    \ref{it:cl11}
 Let $b\in B$. Recall from Lemma~\ref{lem:beta}\ref{it:aka0} that 
  \[
  \alpha_{g,0}h(b)=\bigvee \bigl\{ x\in X\sts g(x)\leq h(b)  \bigr\}.
  \]
  If $\{ x\in X\sts g(x)\leq h(b)  \}=\emptyset$, then $ \alpha_{g,0}h(b)=0$. Since $(0,0)=\bigwedge Z$,
  we can take $b'=0$ in that case.
  Otherwise consider an arbitrary joinand $x\in X$ from above.
  Note that $h(\alpha_h g(x)\wedge b)=g(x)\wedge h(b)=g(x)$ and hence $\alpha_hg(x)\wedge b\geq \beta_hg(x)$.
  Picking any $a\in A$ with
  $(a,b)\in\langle Z\rangle$, we have
  \[
 \langle Z\rangle\ni \Bigl[ \sdp{x}{\alpha_hg(x)}\wedge\sdp{a}{b}\Bigr]\vee\sdp{x}{\beta_hg(x)} = \sdp{x}{\alpha_hg(x)\wedge b}.
   \]
   Taking the join of these elements over all $x$ with $g(x)\leq h(b)$ we obtain
   $(\alpha_{g,0}h(b),b')\in \langle Z\rangle$ for some $b'\leq b$. 

 \ref{it:cl13}
 By Lemma~\ref{lem:beta}\ref{it:aka0} and~\ref{it:bkb0} the formulas for $\alpha_0$ and $\beta_0$ are dual to each
 other. Hence the proof of~\ref{it:cl13} is just the dual of~\ref{it:cl11} after swapping first and second
 components.   
\end{proof}
 
The key technical step in our proof of Theorem~\ref{thm:fiber} is to establish the next claim.

 \begin{claim}
 \label{cl:2}
 The following hold for every $k\in\N$:
\begin{align}
 & \forall b\in G_{Y,k}\colon \sdp{\alpha_{g,k}h(b)}{b}\in\langle Z\rangle,  \label{eq:2k} \tag{2k} \\   
 & \forall a\in H_{X,k}\colon \sdp{a}{\beta_{h,k}g(a)}\in\langle Z\rangle. \label{eq:2k+1} \tag{2k+1}
\end{align}
\end{claim}
  
\begin{proof} 
 We use induction on the index of the statements.

 {\bf Base case:}
 To prove statement (0) let $k=0$ and $b\in G_{Y,0} = Y$. Pick $a\in G_{X,0} = X$ such that $(a,b)\in Z$.
 (The existence of such $a, b$ is guaranteed by the remarks immediately following the definition of $Z$ above.)
 Then $g(a) = h(b)$ implies $a\leq\alpha_{g,0}h(b)$ by Lemma~\ref{lem:beta}\ref{it:aka0}.
 By Claim \ref{cl:1}\ref{it:cl11} we also have $b'\leq b$ such that $(\alpha_{g,0}h(b),b')\in\langle Z\rangle$.
 It follows that
 \[
   \langle Z\rangle\ni \sdp{a}{b}\vee\sdp{\alpha_{g,0}h(b)}{b'} = \sdp{\alpha_{g,0}h(b)}{b}
 \]
 as required.

 {\bf Induction step \eqref{eq:2k}$\Rightarrow$\eqref{eq:2k+1}:}
 Let $a\in H_{X,k} = G_{X,k}^\wedge$ and write $a=\bigwedge T$ for $T\subseteq G_{X,k}$.
 Let $d\in g(T)$.
 By  \eqref{eq:2k} 
 we have
\begin{equation} \label{eq:akbk}
 \langle Z\rangle\ni
\bigwedge \Bigl\{ \sdp{\alpha_{g,k}h(w)}{w} \sts w\in G_{Y,k},h(w)\geq d\Bigr\}=
 \sdp{\alpha_{g,k}(d)}{\beta_{h,k}(d)}.
\end{equation}
 To see that the above equality holds, first note that in the second component we simply have~\eqref{eq:albealt}.
 For the first component of~\eqref{eq:akbk} note that $h(w)\geq d$ implies $\alpha_{g,k}h(w)\geq \alpha_{g,k}(d)$ by
 Lemma~\ref{lem:beta}\ref{it:mon}. Also, since $d\in g(T)\subseteq g(G_{X,k})=h(G_{Y,k})$ (see \ref{it:G3}),
 there exists $w\in G_{Y,k}$ with $h(w)=d$ and the equality in the first component of~\eqref{eq:akbk} follows.
 
 We take the meet over all elements in~\eqref{eq:akbk} for $d$ in $g(T)$ and use Claim \ref{cl:1}\ref{it:cl13}
 to obtain $a'\geq a$ such that
\begin{equation} \label{eq:a''bk}
  \langle Z\rangle\ni \bigwedge\Bigl\{ \sdp{\alpha_{g,k}(d)}{\beta_{h,k}(d)} \sts d\in g(T)\Bigr\} \wedge
  \sdp{a'}{\beta_{h,0}g(a)}.
\end{equation}
 The second component in the above meet is
 \[
  \bigwedge\beta_{h,k} g(T)\wedge \beta_{h,0}g(a)=\beta_{h,k}\bigl( \bigwedge g(T)\bigr)=\beta_{h,k}g(a)
 \]
 by Lemma~\ref{lem:betak}\ref{it:betakwedge} (note that $\beta(P) \subseteq H_{Y,0}$ by~\ref{it:G1}).

The first component of the element in~\eqref{eq:a''bk} is $a'':= \bigwedge_{t \in T} \alpha_{g,k}g(t) \wedge a'$.
Let $t \in T$. Then Lemma~\ref{lem:beta}\ref{it:mon} and $g(t) \geq g(\bigwedge T)$ imply
\[
  \alpha_{g,k} g(t) \geq \alpha_{g,k}g(\bigwedge T) = \alpha_{g,k}g(a) \geq a.
\]
Thus $\alpha_{g,k} g(t) \geq a$ for all $t \in T$ and $a' \geq a$, which yield
$\bigwedge_{t \in T} \alpha_{g,k}g(t) \wedge a' \geq a$.
We conclude
\[
\sdp{a''}{\beta_{h,k}g(a)}\in\langle Z\rangle \quad \text{and}\quad a''\geq a.
\]

 Now let $b\in G_{Y,k}$ with $(a,b)\in\langle Z\rangle$.
 Then $b\geq \beta_{h,k}g(a)$ and so
 \[
   \langle Z\rangle \ni \sdp{a}{b}\wedge\sdp{a''}{\beta_{h,k}g(a)} = \sdp{a}{\beta_{h,k}g(a)}.
 \]

 {\bf Induction step }(2k-1)$\Rightarrow$\eqref{eq:2k}{\bf:}
 This is dual to the proof of \eqref{eq:2k} $\Rightarrow$ \eqref{eq:2k+1}.
 For completeness here is the whole argument verbatim except for switching first and second components, meets and
 joins, as well as the replacements
\[  \left.\begin{array}{l}
   H_{X,k},\alpha_{g,k} \\
          G_{X,k},\beta_{g,k} \end{array} \right\} \to
  \left\{ \begin{array}{l}
   G_{Y,k},\beta_{h,k} \\
   H_{Y,k-1},\alpha_{h,k-1}. \end{array} \right.
\]
The shift of indices in the last part occurs since $H_{X,k} = G_{X,k}^\wedge$ but $G_{Y,k} = H_{Y,k-1}^\vee$ in line with
Remark~\ref{re:duality}.
 
 Let $b\in G_{Y,k} = H_{Y,k-1}^\vee$ and write $b=\bigvee T$ for $T\subseteq H_{Y,k-1}$.
 Let $d\in h(T)$.
 By 
 statement (2k-1)
 we have
\begin{equation} \label{eq:akbk-1}
 \langle Z\rangle\ni
\bigvee \Bigl\{ \sdp{w}{\beta_{h,k-1}g(w)}\sts w\in H_{X,k-1},g(w)\leq d\Bigr\}=
 \sdp{\alpha_{g,k}(d)}{\beta_{h,k-1}(d)}.
\end{equation}
 To see that the above equality holds, first note that in the first component we simply have~\eqref{eq:albealt}.
 For the second component of~\eqref{eq:akbk-1} note that $g(w)\leq d$ implies $\beta_{h,k-1}g(w)\leq \beta_{h,k-1}(d)$ by
 Lemma~\ref{lem:beta}\ref{it:mon}. Also, since $d\in h(T)\subseteq h(H_{Y,k-1})=g(H_{X,k-1})$ (see \ref{it:G3}),
 there exists $w\in H_{X,k-1}$ with $g(w)=d$ and the equality in the second component of~\eqref{eq:akbk-1} follows.

 We take the join over all elements in~\eqref{eq:akbk-1} for $d$ in $h(T)$ and use Claim \ref{cl:1}\ref{it:cl11}
 to obtain $b'\leq b$ such that
\begin{equation} \label{eq:akb''}
 \langle Z\rangle\ni \bigvee\Bigl\{ \sdp{\alpha_{g,k}(d)}{\beta_{h,k-1}(d)} \sts d\in h(T)\Bigr\} \vee
 \sdp{\alpha_{g,0}h(b)}{b'}.
\end{equation}
 The first component in the above join is
 \[
 \bigvee  \alpha_{g,k}h(T)\vee \alpha_{g,0}h(b)=\alpha_{g,k}\bigl( \bigvee h(T)\bigr)=\alpha_{g,k}h(b)
 \]
 by (the dual of) Lemma~\ref{lem:betak}\ref{it:betakwedge}.
 Denote the second component of the element in~\eqref{eq:akb''} by $b''$.
 For $t\in T\subseteq H_{Y,k-1}$ let $d:=h(t)$.
 Then $b \geq t\geq \beta_{h,k-1}(d)$ by the definition of $\beta_{h,k-1}$ in~\eqref{eq:bela}.
So we conclude
 \[
 \sdp{\alpha_{g,k}h(b)}{b''}\in\langle Z\rangle \quad \text{and}\quad b''\leq b.
 \]
 Now let $a\in G_{X,k}$ with $(a,b)\in\langle Z\rangle$.
 Then $a\leq \alpha_{g,k}h(b)$ and so
 \[
   \langle Z\rangle \ni \sdp{a}{b}\vee\sdp{\alpha_{g,k}h(b)}{b''} = \sdp{\alpha_{g,k}h(b)}{b}.
 \]
 This completes the induction for Claim~\eqref{cl:2}.
\end{proof}
 
 We can now return to the proof of Theorem \ref{thm:fiber}.
 Let $(a,b)\in C$ be arbitrary.
 Then there exists $k\in\N$ such that $a\in H_{X,k}$, $b\in H_{Y,k}$.
 Using statement~\eqref{eq:2k+1} of Claim~\ref{cl:2} and its symmetric version with swapped components
 \[
 \langle Z\rangle \ni \sdp{a}{\beta_{h,k}g(a)}\vee\sdp{\beta_{g,k}h(b)}{b} = \sdp{a}{b}.
 \]
 Thus $C=\langle Z\rangle$ as required.
 \end{proof}

%
\section{A finitely generated fiber product with unbounded homomorphisms}\label{sec:c}

\begin{proof}[Proof of Theorem~\ref{thm:exexists}]
 We start with the lattice $L$ of subspaces of the $3$-dimensional vector space over the field with $2$ elements.
 Labelling its elements  
\[
\{0,1\}\cup\{a_i,b_i\sts i=1,\dots,7\},
\]
 we obtain the non-trivial comparisons
\[
a_i \leq b_k\quad \Leftrightarrow\quad
k=i,i+1,i+3 \pmod{7}.
\]
 See Figure~\ref{newlatfig} for a graphical representation.

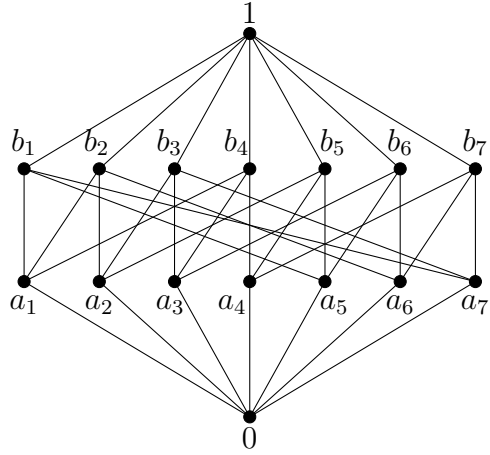
\begin{figure}
\begin{center}
\begin{tikzpicture}
\coordinate (0) at (4,-1.8);
\coordinate (1) at (4,3.3);
\draw[fill] (0) circle [radius=0.8mm];
\draw (0) node [below] {$0$};
\draw[fill] (1) circle [radius=0.8mm];
\draw (1) node [above] {$1$};

\foreach \c in {1,2,3,4,5,6,7}
{\coordinate (a\c) at (\c,0); 
  \coordinate (b\c) at (\c,1.5); 
  \draw[fill] (a\c) circle [radius=0.8mm];
  \draw[fill] (b\c) circle [radius=0.8mm]; 
   \draw (a\c)--(b\c);
   \draw (a\c)--(0);
   \draw (b\c)--(1);}

\foreach \c/\d in {1/2,2/3,3/4,4/5,5/6,6/7,7/1}
  {\draw (a\c)--(b\d); }
\foreach \c/\d in {1/4,2/5,3/6,4/7,5/1,6/2,7/3}
  {\draw (a\c)--(b\d); }
  
 \draw (a1) node [below=1pt] {$a_1$};
 \draw (a2) node [below=1pt] {$a_2$};
 \draw (a3) node [below left=1pt and -8pt] {$a_3$};
 \draw (a4) node [below left=1pt and -3pt] {$a_4$};
  \draw (a5) node [below right=1pt and -6pt] {$a_5$};
  \draw (a6) node [below=1pt] {$a_6$};
  \draw (a7) node [below=1pt] {$a_7$};
 
 \draw (b1) node [above=1pt] {$b_1$};
 \draw (b2) node [above left=1pt and -8pt] {$b_2$};
 \draw (b3) node [above left=1pt and -7pt] {$b_3$};
 \draw (b4) node [above left=1pt and -4pt] {$b_4$};
 \draw (b5) node [above right=1pt and -6pt] {$b_5$};
 \draw (b6) node [above=1pt] {$b_6$};
 \draw (b7) node [above=1pt] {$b_7$};  
  
\end{tikzpicture}
\end{center}
\caption{The lattice $L$.}
\label{newlatfig}
\end{figure}

Next, we `expand' $L$ to an infinite lattice $M$ by `inflating' each $a_i$, $b_i$ to an infinite chain isomorphic to $\omega=\{0<1<2<\dots\}$.
Specifically, the elements are
\[
M:=\{0,1\}\cup \bigcup_{i=1}^7 A_i\cup \bigcup_{i=1}^7 B_i,
\]
where
\[
A_i:=\{a_{i,j}\sts j=0,1,\dots\} \quad \text{and}\quad B_i:=\{b_{i,j}\sts j=0,1,\dots\}
\]
and the comparisons are as follows.
First, each $A_i$, $B_i$ is an increasing chain, i.e.
$a_{i,j}\leq a_{i,\ell}$ and  $b_{i,j}\leq b_{i,\ell}$ for all $j\leq \ell$.
The elements from different $A_i$ are incomparable, as are the elements from different $B_i$.
Comparisons exist between elements of $A_i$ and $B_k$ if and only if $a_i\leq b_k$ in $L$, 
i.e. if and only if $k=i,i+1,i+3\pmod{7}$, and they are given by
\[
a_{i,j}\leq b_{k,\ell} \Leftrightarrow
\begin{cases}
k=i,i+1,i+3\!\!\!\pmod{7} \text{ and }  j\leq \ell;
\text{or }\\
k=i,i+3 \!\!\!\pmod{7} \text{ and } j=\ell+1.
 \end{cases}
\]

These comparisons are illustrated in Figure \ref{fig1}.
Again, it is easy to verify that $M$ is a lattice.
Moreover we claim that
\begin{equation} \label{eq:Mfg}
  M \text{ is generated by the finite set } \{ a_{i,0},a_{i,1}\sts i=1,\dots,7\}.
\end{equation}  
 This follows by a straightforward induction on $j = 0,1,\dots$ using that
\[
b_{i,j}=a_{i-1,j}\vee a_{i,j} \text{ and } a_{i,j+2}=b_{i,j+1}\wedge b_{i+3,j+1}.
\]

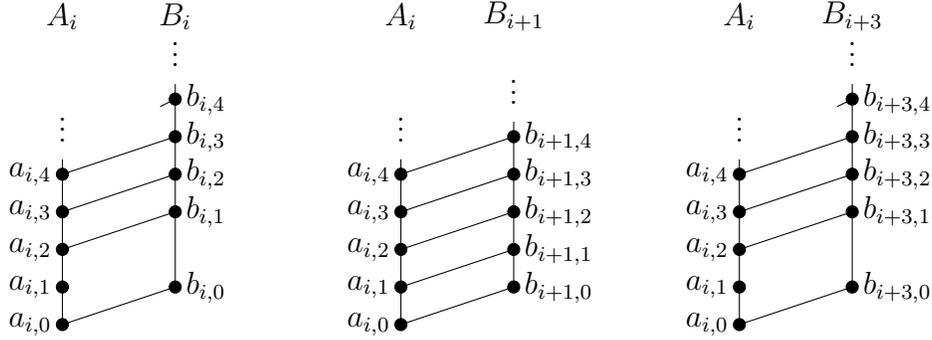
\begin{figure}

\begin{center}

\begin{tikzpicture}


\draw (0,-0.5)  -- (0,1.7);
\draw[fill] (0,-0.5) circle[radius=0.8mm] node [left] {$a_{i,0}$};
\draw[fill] (0,0) circle[radius=0.8mm] node [left] {$a_{i,1}$};
\draw[fill] (0,0.5) circle[radius=0.8mm] node [left] {$a_{i,2}$};
\draw[fill] (0,1) circle[radius=0.8mm] node [left] {$a_{i,3}$};
\draw[fill] (0,1.5) circle[radius=0.8mm] node [left] {$a_{i,4}$};
\draw (0,2.2) node {$\vdots$};

\draw (1.5,0)  -- (1.5,2.7);
\draw[fill] (1.5,0) circle[radius=0.8mm] node [right] {$b_{i,0}$};
\draw[fill] (1.5,1) circle[radius=0.8mm] node [right] {$b_{i,1}$};
\draw[fill] (1.5,1.5) circle[radius=0.8mm] node [right] {$b_{i,2}$};
\draw[fill] (1.5,2) circle[radius=0.8mm] node [right] {$b_{i,3}$};
\draw[fill] (1.5,2.5) circle[radius=0.8mm] node [right] {$b_{i,4}$};
\draw (1.5,3.2) node {$\vdots$};

\draw (0,-0.5) -- (1.5,0);
\draw (0,0.5) -- (1.5,1);
\draw (0,1) -- (1.5,1.5);
\draw (0,1.5) -- (1.5,2);

\draw (1.3,2.4) -- (1.5,2.5);

\draw (0,3.6) node {$A_i$};

\draw (1.5,3.6) node {$B_i$};


\draw (4.5,-0.5)  -- (4.5,1.7);
\draw[fill] (4.5,-0.5) circle[radius=0.8mm] node [left] {$a_{i,0}$};
\draw[fill] (4.5,0) circle[radius=0.8mm] node [left] {$a_{i,1}$};
\draw[fill] (4.5,0.5) circle[radius=0.8mm] node [left] {$a_{i,2}$};
\draw[fill] (4.5,1) circle[radius=0.8mm] node [left] {$a_{i,3}$};
\draw[fill] (4.5,1.5) circle[radius=0.8mm] node [left] {$a_{i,4}$};
\draw (4.5,2.2) node {$\vdots$};

\draw (6,0)  -- (6,2.2);
\draw[fill] (6,0) circle[radius=0.8mm] node [right] {$b_{i+1,0}$};
\draw[fill] (6,0.5) circle[radius=0.8mm] node [right] {$b_{i+1,1}$};
\draw[fill] (6,1) circle[radius=0.8mm] node [right] {$b_{i+1,2}$};
\draw[fill] (6,1.5) circle[radius=0.8mm] node [right] {$b_{i+1,3}$};
\draw[fill] (6,2) circle[radius=0.8mm] node [right] {$b_{i+1,4}$};
\draw (6,2.7) node {$\vdots$};

\draw (4.5,-0.5) -- (6,0);
\draw (4.5,0) -- (6,0.5);
\draw (4.5,0.5) -- (6,1);
\draw (4.5,1) -- (6,1.5);
\draw (4.5,1.5) -- (6,2);

\draw (4.5,3.6) node {$A_i$};

\draw (6,3.6) node {$B_{i+1}$};


\draw (9,-0.5)  -- (9,1.7);
\draw[fill] (9,-0.5) circle[radius=0.8mm] node [left] {$a_{i,0}$};
\draw[fill] (9,0) circle[radius=0.8mm] node [left] {$a_{i,1}$};
\draw[fill] (9,0.5) circle[radius=0.8mm] node [left] {$a_{i,2}$};
\draw[fill] (9,1) circle[radius=0.8mm] node [left] {$a_{i,3}$};
\draw[fill] (9,1.5) circle[radius=0.8mm] node [left] {$a_{i,4}$};
\draw (9,2.2) node {$\vdots$};

\draw (10.5,0)  -- (10.5,2.7);
\draw[fill] (10.5,0) circle[radius=0.8mm] node [right] {$b_{i+3,0}$};
\draw[fill] (10.5,1) circle[radius=0.8mm] node [right] {$b_{i+3,1}$};
\draw[fill] (10.5,1.5) circle[radius=0.8mm] node [right] {$b_{i+3,2}$};
\draw[fill] (10.5,2) circle[radius=0.8mm] node [right] {$b_{i+3,3}$};
\draw[fill] (10.5,2.5) circle[radius=0.8mm] node [right] {$b_{i+3,4}$};
\draw (10.5,3.2) node {$\vdots$};

\draw (9,-0.5) -- (10.5,0);
\draw (9,0.5) -- (10.5,1);
\draw (9,1) -- (10.5,1.5);
\draw (9,1.5) -- (10.5,2);

\draw (10.3,2.4) -- (10.5,2.5);

\draw (9,3.6) node {$A_i$};

\draw (10.5,3.6) node {$B_{i+3}$};

\end{tikzpicture}

\caption{Comparisons between $A_i$ and $B_i$, $B_{i+1}$, $B_{i+3}$.}
\label{fig1}
\end{center}
\end{figure}

Now consider the mapping
\[
h : M\rightarrow L,\ 0\mapsto 0,\ 1\mapsto 1,\ a_{i,j}\mapsto a_i,\ b_{i,j}\mapsto b_i \text{ for } i=1,\dots,7, j=0,1,\dots
\]
Clearly $h$ is a surjective lattice homomorphism with
\[ \ker h=\Bigl\{ \sdp{0}{0},\sdp{1}{1}\Bigr\} \cup \bigcup_{i=1}^7 (A_i\times A_i) \cup \bigcup_{i=1}^7 (B_i\times B_i).
\] 
Furthermore, $h$ is not bounded, since none of its kernel classes $A_i,B_i$ have maximal elements.
We claim that
\begin{equation} \label{eq:khfg}
 \ker h \text{ is finitely generated as a sublattice of } M\times M.
\end{equation}  
To prove this, consider
\[
C_1:=\Bigl\{ \sdp{0}{0},\sdp{1}{1}\Bigr\} \cup
\bigcup_{i=1}^7 \bigl(\{a_{i,0}\}\times A_i\bigr)\cup\bigcup_{i=1}^7 \bigl(\{b_{i,0}\}\times B_i\bigr)
\subseteq \ker h.
\]
Since
$\{ 0,1\}\cup\{a_{i,0},b_{i,0}\sts i=1,\dots,7\}$ is isomorphic to $L$,
it follows that $C_1$ is a lattice isomorphic to $M$. In particular, $C_1$ is finitely generated by~\eqref{eq:Mfg}.
By symmetry, 
\[
C_2:=\Bigl\{ \sdp{0}{0},\sdp{1}{1}\Bigr\} \cup
\bigcup_{i=1}^7 (A_i\times \{a_{i,0}\})\cup\bigcup_{i=1}^7 (B_i\times \{b_{i,0}\})
\]
is a lattice isomorphic to $M$ and is finitely generated.
Any element from $\ker h$ different from $\sdp{0}{0},\sdp{1}{1}$ has the form
$\sdp{a_{i,j}}{a_{i,k}}$ or $\sdp{b_{i,j}}{b_{i,k}}$. Furthermore
\[
\sdp{a_{i,j}}{a_{i,k}} = \sdp{a_{i,0}}{a_{i,k}}\vee \sdp{a_{i,j}}{a_{i,0}}\in C_1\vee C_2,
\]
and a dual statement holds for $\sdp{b_{i,j}}{b_{i,k}}$.
 Thus $\ker h$ is generated by its finitely generated sublattices $C_1,C_2$, which implies~\eqref{eq:khfg},
 and completes the proof of Theorem \ref{thm:exexists}.
\end{proof}

\section{Fiber products of free lattices}
\label{sec:m}

 The following is our main tool for showing that fiber products are not finitely generated.
 
\begin{lemma} \label{lem:bghka}
 Let $A,B$ be lattices. Assume $A$ is generated by a finite set of join prime elements $X$ and satisfies Whitman's
 condition~\eqref{eq:W}. Let $g\colon A \rightarrow D$, $h\colon B \rightarrow D$ be epimorphisms onto
 a lattice $D$. Then for each finite subset $Z$ of the fiber product
 $$C := \{(a, b) \in A \times B \mid g(a) = h(b)\}$$
 there exists $N\in\N$ such that
\begin{equation} \label{eq:bghka}
 \forall (a,b)\in \langle Z\rangle,\ \forall k\in\N,\ \forall w\in H_{X,k}\colon a\geq w \Rightarrow b\geq\beta_{h,k+N}g(w).
\end{equation}
\end{lemma}

\begin{proof}
 Since $Z$ is finite, Lemma \ref{lem:beta}\ref{it:geqbeta} implies that there exists $N\in\N$ such that for all
 $(a,b)\in Z$ we have $b\geq \beta_{h,N}g(a)$.
 We will show that~\eqref{eq:bghka} holds for this $N$ by induction on the complexity of $(a,b)$ over the generating
 set $Z$.
 For the base case let $(a,b) \in Z$ and $w\in H_{X,k}$ such that $a\geq w$.
 Lemma~\ref{lem:beta}\ref{it:mon},\ref{it:desc} yield
 \[ b\geq \beta_{h,N}g(a)\geq\beta_{h,N}g(w) \geq\beta_{h,k+N}g(w). \]
 
 The inductive step splits into two cases:
 
 {\bf Case 1:}
 {\it $(a,b) = (a_1,b_1) \wedge (a_2,b_2)$, where $(a_1,b_1), (a_2,b_2)\in\langle Z\rangle$.}
 If $a= a_1 \wedge a_2 \geq w$ for $w\in H_{X,k}$, then $a_i\geq w$ for each $i\in \{1,2\}$.
 So the induction hypothesis for~\eqref{eq:bghka} yields
 $b_i\geq  \beta_{h,k+N}g(w)$ for each $i\in \{1,2\}$. Therefore,
 $b = b_1 \wedge b_2 \geq  \beta_{h,k+N}g(w)$, as desired. 
 
 {\bf Case 2:}
 {\it $(a,b) = (a_1,b_1) \vee (a_2,b_2)$, where $(a_1,b_1),(a_2,b_2) \in\langle Z\rangle$.}
 We use a second induction on $k\in\N$. For the base case $k=0$, assume $a \geq w \in H_{X,0}$.
 Then $w = \bigwedge W$ for some $\emptyset \neq W \subseteq X$.
 By Whitman's condition~\eqref{eq:W} 
\[
a_1\vee a_2  \geq \bigwedge W\ \Rightarrow\ a_1\geq w \quad \text{or}\quad a_2\geq w \quad \text{or}\quad  a\geq x \text{ for some } x\in W. 
\]
 Since generators $X$ in $A$ are join prime by assumption, the latter case
 yields $a_1\geq x$ or $a_2\geq x$ which implies $a_1\geq w$ or $a_2\geq w$ again.
 Applying the first induction
 assumption (from the induction on term complexity), we find $b_1 \geq \beta_{h,N}g(w)$ or $b_2 \geq \beta_{h,N}g(w)$.
 Therefore, $b = b_1\vee b_2 \geq \beta_{h,N}g(w)$ and the base case is proved.
          
 Next assume $k\geq 1$ and $a\geq w \in H_{X,k}$.
 By definition $w = \bigwedge W$ for some non-empty $W \subseteq H^\vee_{k-1}$.
 By Whitman's condition~\eqref{eq:W}
\begin{equation}
\label{eq:whit}
a_1\vee a_2  \geq \bigwedge W\ \Rightarrow\ a_1\geq w \ \ \text{or}\ \ a_2\geq w \ \ \text{or}\ \  a\geq u \text{ for some } u\in W.
\end{equation}
 The first two alternatives are again straightforward using the first induction assumption on term complexity which
 implies
 $b_1 \geq \beta_{h,k+N}g(w)$ or $b_2 \geq \beta_{h,k+N}g(w)$; in either case, $b = b_1 \vee b_2 \geq \beta_{h,k+N}g(w)$.
 For the third alternative in \eqref{eq:whit} recall that $u = \bigvee U$ for some non-empty $U\subseteq H_{k-1}$.
 For each $v\in U$, we have $a\geq v$ and hence $b\geq\beta_{h,k-1+N}g(v)$ by the second induction hypothesis
 (induction on $k$). Thus
      \begin{align*}
        b &\geq \bigvee \{ \beta_{h,k-1+N}g(v) \sts v\in U \}&& \\
        & \geq \beta_{h,k+N} \bigl(\bigvee\{ g(v) \sts v\in U \}\bigr) &&\text{by Lemma~\ref{lem:beta}\ref{it:Vbk}} \\
        & = \beta_{h,k+N} g(u) &&\\
        & \geq \beta_{h,k+N} g(w) &&\text{by } u \geq w \text{ and Lemma~\ref{lem:beta}\ref{it:mon}}. 
      \end{align*}
 This concludes the induction on $k$ and the proof of~\eqref{eq:bghka}.
    \end{proof}  

\begin{lemma} \label{lem:fl}
 Let $A,B$ be lattices. Assume $A$ is generated by a set of join prime elements $X$ and satisfies Whitman's
 condition~\eqref{eq:W}. Let $g\colon A \rightarrow D$, $h\colon B \rightarrow D$ be epimorphisms onto
 a lattice $D$.

 If the fiber product of $g$ and $h$ is a finitely generated sublattice of $A\times B$, then $h$ is lower bounded.
\end{lemma}

\begin{proof}
 Using contraposition we assume that $h$ is not lower bounded.
 Then we have $d\in D$ such that $h^{-1}(d)$ does not have a least element.

 Fix a finite subset $Z \subseteq C$ and let $N$ be as in Lemma~\ref{lem:bghka} such that~\eqref{eq:bghka} holds.
 Let $k\in\N$ such that $g^{-1}(d)\cap H_{X,k} \neq\emptyset$; such $k$ exists since $g$ is surjective and
 $A = \bigcup_{k\in\N} H_{X,k}$. Let $a\in g^{-1}(d)\cap H_{X,k}$. Since $h^{-1}(d)$ has no least element,
 there exists $b\in h^{-1}(d)$ such that $b < \beta_{h,k+N}(d)$. 
 Then $(a,b) \in C$ but $(a,b)\not\in\langle Z\rangle$ by Lemma~\ref{lem:bghka}.
 Since $Z$ was an arbitrary finite subset of $C$, this proves that $C$ is not finitely generated.
\end{proof}

 We are now in a position to prove Theorem~\ref{thm:prime}.

\begin{proof}[Proof of Theorem~\ref{thm:prime}] 
 Assume the fiber product of $g$ and $h$ is a finitely generated sublattice of $A\times B$.
 Since $A$ and $B$ are generated by join prime elements by assumption, $h$ and $g$ are lower bounded by
 Lemma~\ref{lem:fl}. Moreover, since $A$ and $B$ are also generated by meet prime elements, the dual of
 Lemma~\ref{lem:fl} yields that $h$ and $g$ are upper bounded as well.
\end{proof}

\section{Deciding bounded lattices}
\label{sec:algorithm}

 It is known to be decidable whether a finitely presented lattice is bounded by an unpublished result of Freese and
 Nation; see~\cite[page 251]{MR1319815}.
 We give a proof for this and that it is decidable whether a finitely generated sublattice satisfying
 Dean's condition~\eqref{eq:D} of a finitely presented lattice is bounded.

 Let $P$ be a finite partial lattice, and let $n\in\N$. Then $S := P^{(\vee\wedge)^n\vee}$ is a finite join-subsemilattice
 of $F(P)$ with the join of the empty set, i.e. $\bigwedge P$, as its least element. Because $S$ is join closed,
 has a least element and is finite, any $a,b\in S$ have an infimum $\inf(a,b)\in S$.
 Note that $\inf(a,b) \leq a\wedge b$ where the latter denotes the meet in $F(P)$; equality may hold e.g.
 if that meet happens to be defined in the partial lattice $P$. Hence $(S,\inf,\vee)$ is a finite lattice
 but not necessarily a sublattice of $F(P)$. Instead $(S,\inf,\vee)$ turns out to be a homomorphic image of $F(P)$.

 By~\cite[Lemma 2-6.11]{FN:FFP} and the subsequent discussion in the extended version of that paper,
 the~\emph{standard homomorphism}
\[ f\colon F(P) \to S,\ d \mapsto \bigvee\{ w\in S \sts w\leq d \}, \]
 exists and is a lower bounded epimorphism. For any $d\in S\subseteq F(P)$ we have $f(d) = d$ and consequently
 $d = \beta_f(d)$.
  
\begin{lemma} \label{lem:Pvee}
 Let $A$ be a lattice with finite generating set $X$, let $P$ be a finite partial lattice,
 and let $g\colon A\to F(P)$ be a homomorphism.
 Assume that $g(A)$ satisfies Dean's condition~\eqref{eq:D} for the generating set $g(X)$ and
 $g(X)\subseteq P^{(\vee\wedge)^n\vee}$ for $n\in\N$. 
 Then $g$ is lower bounded if and only if its composition $fg\colon A\to P^{(\vee\wedge)^n\vee}$ with the standard
 homomorphism $f$ is lower bounded.
\end{lemma}

\begin{proof}
 The forward direction follows since the composition of bounded homomorphisms is bounded.

 For the backward direction, assume that $fg$ is lower bounded.
 Let $d\in P^{(\vee\wedge)^n\vee}\cap g(A)$. Then $f(d) = d$ yields $\beta_{fg}(d) = \beta_{g}(d)$.
 Hence $g^{-1}(d)$ has a least element for any $d\in P^{(\vee\wedge)^n\vee}\cap g(A)$.
 In particular $\beta_gg(x)$ exists for any generator $g(x)$ of $g(A)$.
 Thus $g\colon A\to F(P)$ is lower bounded by Corollary~\ref{cor:lb}.
\end{proof}

 We can now give the algorithm for deciding boundedness that proves Theorem~\ref{thm:decidable}.

\begin{proof}[Proof of Theorem~\ref{thm:decidable}]
 For $D=F(P)$ finitely presented, $D$ is lower bounded if and only if the lattice $S := P^\vee$ is lower bounded by
 Lemma~\ref{lem:Pvee} with $A$ the free lattice over the set $P$ and $g\colon A\to D$ the natural epimorphism.

 In case $D$ is generated by some finite subset $X$ of $F(P)$ and satisfies Dean's condition~\eqref{eq:D}, 
 assume $X\subseteq P^{(\vee\wedge)^n\vee}$ for some $n\in\N$. Then $D$ is lower bounded if and only if the sublattice
 $S$ of $P^{(\vee\wedge)^n\vee}$ that is generated by $X$ is lower bounded by Lemma~\ref{lem:Pvee}
 with $A$ the free lattice over $X$ and the natural epimorphism $g\colon A\to D$.

 In either case it suffices to decide whether the finite lattice $S$ is lower bounded.
 This can be done in time $O(|S|^2)$ by~\cite[Theorem 11.20]{MR1319815}.
 Note that $|S|$ is at most exponential in the size of the input $P$, $X$, respectively. Hence we can decide
 whether $D$ is bounded in exponential time.
\end{proof}

 For the second case in Theorem~\ref{thm:decidable} we note that a sublattice of $F(P)$ trivially satisfies
 Dean's condition~\eqref{eq:D} if $F(P)$ satisfies Whitman's condition~\eqref{eq:W}.
 By Dean's solution to the word problem for $F(P)$~\cite[Theorem 2-3.4]{FN:FFP}
 this is equivalent to $P$ satisfying Whitman's condition~\eqref{eq:W} whenever meets and joins are defined in $P$.
 In other words, $F(P)$ fails~\eqref{eq:W} if and only if there is a failure in $P$ using the defined joins and meets.

\section*{Acknowledgments}
  
 We thank Ralph Freese and J.~B. Nation for discussions on the material in this paper as well as an 
 anonymous referee for suggestions on the presentation.

\bibliographystyle{plain}

\begin{thebibliography}{1}
\bibitem{BM:SFP}
M.~R. Bridson and C.~F. Miller, III.
\newblock Structure and finiteness properties of subdirect products of groups.
\newblock {\em Proc. London Math. Soc. (3)}, 98(3):631--651, 2009.

\bibitem{AC19}
 A. Clayton.
\newblock On finitary properties for fiber products of free semigroups and free monoids.
\newblock 2019. \verb+arXiv:1907.01378+


\bibitem{FN:FFP}
R. Freese and J.~B. Nation.
\newblock Free and finitely presented lattices.
\newblock In {\em Lattice theory: special topics and applications. {V}ol. 2},
  pages 27--58. Birkh\"{a}user/Springer, Cham, 2016.
\newblock Extended version available on \\ \verb+http://math.hawaii.edu/~ralph/Preprints/FreeLat-ExtendedVersion.pdf+
  
\bibitem{MR1319815}
R. Freese, J. Je{\v{z}}ek, and J.~B. Nation.
\newblock {\em Free lattices}, volume~42 of {\em Mathematical Surveys and
  Monographs}.
\newblock American Mathematical Society, Providence, RI, 1995.

\bibitem{MR0472614}
B.~J{\'o}nsson and J.~B. Nation.
\newblock A report on sublattices of a free lattice.
\newblock In {\em Contributions to universal algebra ({C}olloq., {J}\'ozsef
  {A}ttila {U}niv., {S}zeged, 1975)}, pages 223--257. Colloq. Math. Soc.
  J\'anos Bolyai, Vol. 17. North-Holland, Amsterdam, 1977.

\bibitem{MR:GSP}
P. Mayr and N. Ru{\v{s}}kuc.
\newblock Generating subdirect products.
\newblock {\em J. Lond. Math. Soc.}, 100(2):404--424, 2019.

\bibitem{MR:FPD}
P. Mayr and N. Ru\v{s}kuc.
\newblock Finiteness properties of direct products of algebraic structures.
\newblock {\em J. Algebra}, 494:167--187, 2018.

\bibitem{MR0313141}
R. McKenzie.
\newblock Equational bases and nonmodular lattice varieties.
\newblock {\em Trans. Amer. Math. Soc.}, 174:1--43, 1972.

\end{thebibliography}


\end{document}